\documentclass[letterpaper, 10 pt, conference]{ieeeconf}  % Comment this line out
\IEEEoverridecommandlockouts                              % This command is only
\usepackage{graphicx}      % include this line if your document contains figures
\graphicspath{./Matlab/}	
\usepackage[usenames,dvipsnames]{color}

\usepackage{mathtools}
\usepackage{cite}
\usepackage{amssymb}
\usepackage[mathscr]{eucal}
\usepackage{units}
\usepackage{yhmath}
\usepackage{units}
\usepackage[utf8]{inputenc}
\usepackage{balance}
\usepackage{dsfont}
\usepackage{caption}
\usepackage{subcaption}
\usepackage{dblfloatfix}
\usepackage{epstopdf}
\usepackage{color} % change text color        
\usepackage{flushend}

\usepackage{tikz}
\usetikzlibrary{decorations.pathmorphing} % for snake lines
\usetikzlibrary{matrix} % for block alignment
\usetikzlibrary{arrows} % for arrow heads
\usetikzlibrary{calc} % for manimulation of coordinates

\tikzstyle{block} = [draw,rectangle,thick,minimum height=2em,minimum width=2em]
\tikzstyle{sum} = [draw,circle,inner sep=0mm,minimum size=2mm]
\tikzstyle{connector} = [->,thick]
\tikzstyle{line} = [thick]
\tikzstyle{branch} = [circle,inner sep=0pt,minimum size=1mm,fill=black,draw=black]
\tikzstyle{guide} = []
\tikzstyle{snakeline} = [connector, decorate, decoration={pre length=0.2cm,
                         post length=0.2cm, snake, amplitude=.4mm,
                         segment length=2mm},thick, magenta, ->]

\tikzset{
  arm/.default=1cm,
  arm/.code={\def\myarm{#1}}, % store value in \myarm
  angle/.default=0,
  angle/.code={\def\myangle{#1}} % store value in \myangle
}

\usepackage{blkarray}% http://ctan.org/pkg/blkarray
\newtheorem{theorem}{Theorem}

\newtheorem{remark}{Remark}

\DeclareMathOperator{\diag}{diag}
\DeclareMathOperator{\tr}{tr}

\newcommand{\R}{\ensuremath{\mathds{R}}}

\newcommand{\bi}[0]{\ensuremath{\mathbf{i}}}
\newcommand{\bbm}{\begin{bmatrix}}
\newcommand{\ebm}{\end{bmatrix}}

\newboolean{showcomments}
\setboolean{showcomments}{true}
\newcommand{\enrique}[1]{\ifthenelse{\boolean{showcomments}}
{\textcolor{Red}{(Enrique says: #1)}}{}}
\newcommand{\addcite}[0]{\ifthenelse{\boolean{showcomments}}
{\textcolor{Purple}{(add cite(s))}}{}}
\newcommand{\addcites}[0]{\ifthenelse{\boolean{showcomments}}
{\textcolor{Purple}{(add cite(s))}}{}}
\newcommand{\addref}[0]{\ifthenelse{\boolean{showcomments}}
{\textcolor{Purple}{(add ref)}}{}}
\newcommand{\todo}[1]{\ifthenelse{\boolean{showcomments}}
{\textcolor{Purple}{(To Do: #1)}}{}}
\newcommand{\highlight}[1]{\ifthenelse{\boolean{showcomments}}
{\textcolor{blue}{#1}}{#1}}

\title{\LARGE \bf 
iDroop: A dynamic {droop} controller to decouple power grid's steady-state and  dynamic performance
}

\author{Enrique Mallada%,~\IEEEmembership{Member,~IEEE}
\thanks{
The author is with the Department of Electrical and Computer Engineering, Johns Hopkins University, Baltimore, MD 21218 USA (e-mail: mallada@jhu.edu).
This work was supported by NSF CPS grant CNS 1544771, Johns Hopkins E$^2$SHI Seed Grant, and Johns Hopkins WSE startup funds.
}
}
\renewcommand{\baselinestretch}{.918}

\begin{document}

\maketitle
\begin{abstract}                % Abstract of no more than 250 words.
This paper presents a novel Dynam-\emph{i}-c \emph{Droop} (\emph{iDroop}) control mechanism to perform primary frequency control with gird-connected inverters  that improves the network dynamic performance. The work is motivated by the dynamic degradation experienced by the power grid due to the increase in asynchronous inverted-based generation. We show that the widely suggested virtual inertia solution suffers from unbounded noise amplification (infinite $\mathcal H_2$ norm) when measurement noise is considered.  This suggests that virtual inertia could potentially further degrade the grid performance once broadly deployed.

This motivates the proposed solution in this paper that overcomes the limitations of virtual inertia controllers while sharing the same advantages of  traditional droop control. In particular, our iDroop controllers are decentralized, rebalance supply and demand, and provide power sharing. 
Furthermore, our solution can improve the dynamic performance without affecting the steady state solution. 
Our algorithm can be incrementally deployed and can be guaranteed to be stable using a decentralized sufficient stability condition on the parameter values. 
We illustrate several features of our solution using numerical simulations.
\end{abstract}

\section{Introduction}
%Droop control dates from grids origins
Droop control has a long history in power system frequency control~\cite{Ziebolz:1947tq}. 
It is perhaps one of the simplest -- and yet very effective -- decentralized mechanisms that achieve synchronization and supply-demand balance in a network with several generating resources running in parallel~\cite{Almeras:1951vh}. While its implementation may vary depending on the specific device, its basic operational principle remains unchanged: whenever frequency is above (below) the nominal value, decrease (increase) power proportionally to the frequency deviation. Thus it is usually referred as the primary layer of the frequency control architecture~\cite{Machowski:2011uc}.
Not surprisingly, its many benefits have made droop control one of the features of power system engineering that have successfully survived decades of technological advances.
%Even in the context of microgrids, droop control is a widely used solution for controlling parallel microgrid inverters~\addcites. 

%
Unfortunately, the very principles that this mechanism relies on are becoming less and less valid due to several reasons. Firstly, droop control relies on the fact that demand is frequency dependent, yet with the increase of power electronic based loads the aggregate load is becoming less sensitive to frequency~\cite{WoodWollenberg1996}. Secondly, while traditional generation always provides some level of droop control, renewable generation is insensitive to frequency fluctuations. The remaining conventional generators are forced to handle the whole burden of regulating the frequency. Thirdly, newer types of generators have little or no inertia at all when compared with traditional ones, which is slowly introducing a dynamic degradation that concerns many utilities~\cite{Kirby:2005uy}.

 %Droop control ties dynamic and steady state performance.
Recently, a large body of literature has been developed on the design of distributed control mechanisms with the objective to synthetically generate frequency responsive generation and demand that can overcome the diminishing participation of these elements in primary frequency control~\cite{Zhao:2014bp,Zhang:2015kq,Dorfler:2014uu,Li:2014we}. In particular, the proposed strategies seek to introduce a more frequency responsive devices, either on the load side~\cite{Zhao:2014bp,Mallada:2014ui} or on the generation side~\cite{Li:2014we}, that not only implements droop characteristics for primary frequency control, but also guarantee higher level operational constraints such as restoring frequency to nominal value~\cite{mallada2014distributed}, preserving inter area flow constraints,  respecting thermal limits, and providing efficiency\cite{Mallada:2014ui}. While these solutions directly address the loss of frequency responsiveness in the grid, they do not explicitly address dynamic degradation. 

Interestingly, droop control can in principle also provide dynamic performance improvement by properly selecting the droop coefficients~\cite{Motter:2013iw}. For example, in an under-damped power grid, decreasing the droop coefficient can reduce the frequency Nadir (maximum frequency excursion) in the same way increasing friction can reduce the overshoot of an under-damped spring-mass system. However this is not a practical solution since it also requires the generator to take a larger share of the supply-demand imbalance. As a result, the efficiency of the steady-state resource allocation becomes intrinsically coupled with the possible dynamic performance improvement. It is the purpose of this work to eliminate this coupling.

%The goal of this work.
 This paper proposes a novel dynam-{\it i}-c {\it droop} (iDroop) control algorithm that is able to maintain all the desired features of traditional droop control while providing enough design flexibility to improve the dynamic performance. More precisely, we present a control scheme --whose input is frequency and output is power generation-- that preserves the same steady-state characteristics as traditional droop control, maintains grid stability, and provides supply-demand balance. Our iDroop controller can be implemented by inverters providing power from renewable sources or by intelligent loads. 
 %It can be incrementally deployed, and it can be integrated with newly proposed frequency control algorithm Unified Control (UC)~\addcite. 
Finally, we numerical demonstrate that one can use the additional flexibility to improve the $\mathcal H_2$-norm of the system.

%Related work
%This work is motivated by the study in communication networks that use passivity ideas to design a wide class of congestion control protocols~\cite{Wen:2004fa} as well as the use lead-lag compensators to provide delay robustness to congestion control algorithms~\cite{Paganini:gz}. There has been a recent trend in applying similar ideas for frequency regulation~\addcites. In particular the work \addcite considers the use of passivity conditions to characterize a large family of frequency controllers and shows that in several scenarios dynamic controllers can provide more relaxed stability conditions. Our work complements that of \addcite in several ways. Firstly, our focus is on designing an alternative solution to droop control that can have a performance improvement without modifying the steady state. Since not every 

%Paper organization
\emph{Paper Organization:} Section \ref{sec:model} describes the power network model as well as several inverter operational modes used to interface with the grid. Section \ref{sec:steady-vs-dynamic} uses dynamic and steady state performance metrics to motivate the need for a novel droop control solution. Section \ref{sec:idroop} introduces the proposed iDroop control and shows how  our solution is able to preserve the same steady state as droop control while providing enough flexibility to improve dynamic performance.  Moreover, we provide a decentralized sufficient condition on the parameter values that guarantee the stability of our controllers.
%Section \ref{sec:interoperability} shows how our iDroop control can be readily integrated to the new with recently proposed solutions for secondary and tertiary frequency control.
We numerically illustrate the functionalities of iDroop in Section \ref{sec:simulations} and conclude in Section \ref{sec:conclusions}.

\section{Network Model}\label{sec:model}

We consider a power system composed by $n$ buses denoted by either $i$ or $j$, i.e. $i,j\in N:=\{1,\dots, n\}$. We use $ij$ to denote the transmission line that connects bus $i$ and $j$, and use $E$ to denote the set of lines, i.e. $ij\in E$. Thus the topology of the power system is described by the graph $G=(N,E)$.
The admittance of the line $ij$ is given by $y_{ij}=g_{ij}-\bi b_{ij}$, where $g_{ij}$ and $b_{ij}$ denote the conductance and susceptance of line $ij$, respectively. The state of the network is described by the complex voltages $(V_i)_{i\in N}=(v_ie^{j\theta_i})_{i\in N}$ where $v_i$ and $\theta_i$ represent the voltage magnitude and phase of bus $i$, respectively.

\subsection{Generators and Loads}\label{ssec:generator}

We model the dynamics of each conventional generator using the standard swing equations~\cite{Shen:1954eo,Bergen:1981hja}. We denote the frequency of each generator by $\omega_i$ which evolves according to 
\begin{subequations}\label{eq:generation}
\begin{align}
\dot\theta_i &= \omega_i,\\
M_i\dot\omega_i &= p^{in} +q^r_i - (D_i+\frac{1}{R_i^g})\omega_i -P_i^e,
\end{align}
\end{subequations}
where $M_i$ denotes the aggregate generator inertia, $D_i$ is the aggregate damping and frequency dependent load coefficient, and $R_i^g$ is the droop coefficient. $p_i^{in}$ denotes the net constant power injection at bus $i$, $q^r_i$ is the controllable input power injected by grid-connected inverters, and $P_i^e$ denotes the net electric power drawn by the grid.

\subsection{Inverters}
In this paper we seek to develop a new control scheme that can be implemented by  inverters to improve the dynamic performance of the power grid. Since the power electronics of the inverters are significantly faster than the electromechanical dynamics of the generators, we assume that inverters can statically update its power, i.e.
\begin{equation}\label{eq:inverter}
q^r_i = u_i,
\end{equation}
where $u_i$ is the command input. We assume that \eqref{eq:inverter} represents the aggregate power of all the inverters connected at bus $i$.

There are different operational modes in which inverters can be interfaced with the power grid~\cite{Lee:2013fk,CHANDORKAR:1993tp,SimpsonPorco:2013gs,Liu:2016kr}.
Here we briefly review the most common ones:

\noindent
{\it Constant Power:} This is the default operational mode in today's grid and amounts to setting
\begin{equation}\label{eq:inverter-constant-power}
u_i = q_i^{r,0},
\end{equation}
where $q_i^{r,0}$ is a constant parameter representing power generation set point.

\noindent
{\it Droop Control:} This mode aims to change the power injection of the inverter to provide additional droop capabilities by setting
\begin{equation}\label{eq:inverter-constant-power}
u_i = q_i^{r,0} - \frac{1}{R^r_i}\omega_i,
\end{equation}
where $R^r_i$ is the droop coefficient.

\noindent
{\it Virtual Inertia:} This operational mode has been recently proposed~\cite{Beck:eb,Driesen:ft} as an alternative method  to compensate the loss of inertia and is given by
\begin{equation}\label{eq:inverter-constant-power}
u_i  = q_i^{r,0} -\frac{1}{R^r_i}\omega_i - M^v_i\dot\omega_i,
\end{equation}
where $M^v_i$ represents the virtual inertia.

%In the next section we will discuss all these solutions and their effect on the $\mathcal H_2$ performance of the system. In particular, we will show that while some of these solutions can indeed provide performance improvement, in all these cases come at the cost of altering the steady state.

\subsection{Power Flow Model}
We consider a Kron-reduced network model~\cite{Varaiya:ig}, where constant impedance loads are implicitly included in the line impedances of the reduced network. Thus every remaining bus represents a grid generator.  
We further use the DC network model which has been widely adopted for purpose of designing frequency controllers for a long time~\cite{Quazza:1970cd,Quazza:1966vx}.%,Calovic:1972de}. 

%We further make the following assumptions which are widely accepted for the purpose of designing frequency controllers~\addcite. {\it  (i) Lossless:} we assume that the transmission lines are lossless ($g_{ij}=0$). {\it (ii) Decoupling:} we assume that voltages depend primarily on the reactive power and phases on active power~\cite{Kundur:1994tx}.
%(iii) 

Therefore, the  total electric power drained by the network at bus $i$ is
\begin{equation}\label{eq:network}
P_i^{e}  =  -\sum_{j\in \mathcal N_i} b_{ij}(\theta_i-\theta_j),
\end{equation}
where the set $\mathcal N_i$ denotes the set of neighboring buses adjacent to bus $i$.

\subsection{Network Dynamics}
Combining \eqref{eq:generation}-\eqref{eq:network} we arrive to the following compact description of the system dynamics 
\begin{subequations}\label{eq:system}
\begin{align}
\dot\theta &= \omega\label{eq:system-a}\\
M\dot\omega & = p^{in} + q^r - ({R^g}^{-1}+D)\omega - L_B\theta\label{eq:system-b}
\end{align}
\end{subequations}
where $M:=\diag(M_i,\,i\in N)$, $D:=\diag(D_i,\,i\in N)$,  ${R^g}^{-1}=\diag(\frac{1}{R_i^g},\,i\in N)$, $\omega:=(\omega_i,\,i\in N)$, $p^{in}=(p_i^{in},\,i\in N)$, 
$L_B$ is the $b_{ij}$-weighted Laplacian matrix
\begin{equation}
(L_B)_{ij} = \begin{cases}
-b_{ij}, & \text{if $i\not=j$, $ij\in E$}, \\
\sum_{k\in\mathcal N_i} b_{ik},& \text{if $i=j$,}\\
0,&\text{otherwise,}
\end{cases}
\end{equation}
and  $q^r:=(q^r_i,\,\,i\in N)$ is given by
\begin{equation}\label{eq:qr}
q^r_i = \begin{cases}
q_i^{r,0}, & \text{if $i\in CP$,}\\
q_i^{r,0} -\frac{1}{R_i^r}\omega_i,& \text{if $i\in DC$,}\\
q_i^{r,0} -\frac{1}{R_i^r}\omega_i - M^v_i\dot\omega_i,& \text{if $i\in VI$.}\\
\end{cases}
\end{equation}
The sets $CP$, $DC$ and $VI$ are the subsets of buses that have inverters operating in constant power, droop control and virtual inertia modes, respectively. 
%The parameters $M_i^v$ and $R_i^R$ represent the virtual inertia and droop coefficient of the inverter $i$, respectively.

%\todo{fully characterize the steady state behavior}
In the absence of higher layer controllers, such as automatic generation control~\cite{deMello:1973jy}, the system can synchronize with a nontrivial frequency deviation from the nominal ($\omega_i^*\not=0$).\footnote{We assume here w.l.o.g. that the nominal frequency is $0$}
Thus we refer to the vector $(\theta^*(t):=\theta^*+\omega^*t,\omega^*,{q^r}^*)$ as a steady state solution of the system \eqref{eq:system}, with \eqref{eq:system-b}  equal zero and $\frac{\text{d}}{\text{d}t} \theta^*(t)=\omega^*$. Furthermore, using \eqref{eq:system-b} it is easy to see that $L_B\theta^*(t)$ is constant, which implies that $\omega^*=\mathbf 1_n \omega^*_0$, with $\mathbf 1_n\in \R^n$ being the vector of all ones, and the scalar $\omega_0^*$ given by
\begin{equation}\label{eq:omega_0}
 \omega_0^* = \frac{\sum_{i=1}^n p_i^{in}+{q_i^{r,0}}}{\sum_{i=1}^n (D_i + {R_i^g}^{-1})+\sum_{i\in DC\cup VI} {R_i^r}^{-1} }.
\end{equation}
In summary, the steady state solution of \eqref{eq:system} and \eqref{eq:qr} is given by 
$(\theta^*(t),\omega^*,{q^r}^*)=(\theta^*+\mathbf{1}_n \omega_0^*t, q^{r*} )$, where $q_i^{r*}=q_i^{r,0}$,  if $i\in CP$, or $q_i^{r*}=q_i^{r,0} - {R_i^r}^{-1}\omega_0^*$, if $i\in DC\cup VI$.
Thus the smallest invariant set that includes all possible steady states  is given by
\begin{equation}\label{eq:E}
\mathcal E\!:=\!\{\!(\theta\!+\!\lambda\mathbf 1_n,\omega,q^r)\!\in\!\mathds R^{3n}\!\!:\!\text{$\lambda\in\mathds R$},\eqref{eq:system-b}=0,\omega_i\!=\!\omega_0^*\}. 
\end{equation}

\section{Steady State and Dynamic Performance }\label{sec:steady-vs-dynamic}

In this section we introduce two metrics, one for steady state and another for dynamic performance, and illustrate how the existing  solutions for inverter control either cannot produce  any 
 dynamic performance improvement, or the performance improvement comes at the cost of steady state deviation from the desired operational point.

\subsection{Steady State Performance}\label{ssec:ss-performance}

We now define the steady state performance metric used in this paper. Let $\delta q_i^g:=-\frac{1}{R_i^g}\omega_i$ and $\delta q_i^r:=q_i^r-q_i^{r,0}$ be the power deviation of generators and inverters, respectively. After a disturbance these quantities need to be modified in order to compensate a  supply-demand imbalance, i.e. 
\begin{equation}\label{eq:supply-demand-balance}
\sum_{i=1}^n \delta q_i^g + \sum_{i\in DC\cup VI}\delta q_i^r  = \Delta P,
\end{equation}
where $\Delta P$ denotes the power imbalance. 
We assign to both conventional and inverter-based generators a cost $c_i^g(q_i^g):=\frac{\alpha^g_i}{2}(q_i^g)^2$ and $c_i^r(q_i^r):=\frac{\alpha^r_i}{2}(q_i^r)^2$ respectively.
Thus given a set deviations satisfying \eqref{eq:supply-demand-balance} the total system cost of mitigating an imbalance of $\Delta P$ is given by
\noindent 
\begin{flalign}\label{eq:ss-cost}
&\text{\bf SS-Cost:} &\sum_{i=1}^n \frac{\alpha^g_i}{2}(\delta q_i^g)^2+\!\!\!\!\! \sum_{ i\in DC\cup VI}\!\!\!\frac{\alpha^r_i}{2}(\delta q_i^r)^2\qquad\qquad
\end{flalign}

One of the main attractive features of droop control is its ability to share the supply-demand mismatch among different resources. Recent works have shown that the steady state allocation of droop control can be represented as the solution of an optimization problem (see e.g.~\cite{Zhao:2014bp}). 

In particular, if we define let $\delta d_i=:D_i\omega_i$ denote the frequency dependent demand deviation, then it can be shown that the system \eqref{eq:system} solves:
%\begin{subequations}
%\begin{align}\label{eq:cost-ss}
%\underset{q^g_i,d_i}{\text{minimize}}\quad &\sum_{i=1}^n \frac{R^g_i(q_i^g )^2}{2}+ \frac{(d_i)^2}{2D_i} \\
%\text{subject to}\quad & \sum_{i=1}^n p_i^{in} +q_i^r+ q_i^g - d_i =0
%\end{align}
%\end{subequations}
\begin{subequations}\label{eq:droop-control-problem}
\begin{align}
\underset{\delta q^g_i,\delta q^r_i,\delta d_i}{\text{minimize}}\;\; &\sum_{i=1}^n \frac{R^g_i(\delta q_i^g )^2}{2} \!+\! \frac{(\delta d_i)^2}{2D_i} \!+\!\!\!\!\!\!\!\! \sum_{\quad i\in DC\cup VI}\!\!\!\!\!\!\!\!\frac{R^r_i(\delta q^r_i )^2}{2}\\
\text{subject to}\;\; & \sum_{i=1}^n p_i^{in} \!+\!q_i^{r,0}\!+\! \delta q_i^g \!-\!\delta d_i\!+\!\!\!\!\!\sum_{i\in DC\cup VI}\!\!\!\!\!\delta q_i =0
\end{align}
\end{subequations}
The proof of this claim is already standard in the community (see e.g. ~\cite{Zhao:2014bp,Zhao:2015kk,Zhang:2015kq,Zhang:2015hn,Kasis:2016vn}), we refer the reader to~\cite{Zhao:2014bp}  for a  proof of a similar statement.

As a result of the steady state characteristic of droop control, it is possible to optimally minimize the steady state cost as it is summarized in the next theorem.

\begin{theorem}[Droop Control Optimality]\label{thm:droop-control-optimality}
Let $(\theta^*(t),\allowbreak\omega^*,\allowbreak{q^r}^*)$ be the steady state solution of \eqref{eq:system} where ${q_i^r}^*$ is given by \eqref{eq:qr}. If $R_i^r=\alpha_i^r$ and $R_i^g=\alpha_i^g$, then ($\delta {q_i^g}^* := -\frac{1}{R_i^g}\omega_0^*$, $\delta {q_i^r}^* := {q_i^r}^*-q_i^{r,0}= -\frac{1}{R_i^g}\omega_0^*$) is the unique allocation  that minimizes the steady state cost \eqref{eq:ss-cost} subject to \eqref{eq:supply-demand-balance}, where $\Delta P:=\sum_{i=1}^n p_i^{in}+q_i^{r,0} - D_i\omega_0^*$.
\end{theorem}
%\begin{proof}
%See Appendix.
%\end{proof}
\begin{proof}
We start by characterizing the optimal solution of minimizing \eqref{eq:ss-cost} subject to \eqref{eq:supply-demand-balance}. Since the problem has a strictly convex objective with linear constraints then there is a {\it unique} solution that is characterized by the Karush-Kuhn-Tucker (KKT) conditions for optimality\cite{Boyd:2004cv}. Thus, $\hat{\delta q^g}^*$ and $\hat{\delta q^r}^*$ is an optimal solution if and only if there exists a scalar $\lambda^*$ such that
\begin{equation}\label{eq:KKT}
\alpha_i^g\hat{\delta q_i^g}^*   = \lambda^*\,\forall i,\;\;\;  \alpha_i^r\hat{\delta q_i^r}^*   = \lambda^*\,i\in DC\cup VI,
\end{equation}
and $\sum_{i=1}^n\hat{\delta q_i^g}^*+\hat{\delta q_i^r}^*=\Delta P$.

To finalize the proof we just need to show that when $R_i^r=\alpha_i^r$ and $R_i^g=\alpha_i^g$, then the steady state deviations of droop control (${\delta {q_i^g}^*} := -\frac{1}{R_i^g}\omega_0^*$ and $\delta {q_i^r}^*:= {q_i^r}^*-q_i^{r,0}$) satisfy the KKT conditions.
\begin{flalign*}
&\!\!\!\!\!\!\!\text{By \eqref{eq:qr},}\;{\delta q_i^r}^*:={q_i^r}^* - q_i^{r,0} =
\begin{cases}
 -\frac{1}{R_i^r}\omega_0^*, &\text{if $i\in DC \cup VI$,}\\
 0, & \text{otherwise.}
 \end{cases}
\end{flalign*}
Therefore,  since by definition ${\delta q^g}^*= -\frac{1}{R_i^g}\omega_0^*$, then we have that when $R_i^r=\alpha_i^r$ and $R_i^g=\alpha_i^g$,  $({\delta q^g}^*,{\delta q^r}^*)$ satisfies \eqref{eq:KKT} for $\lambda^*=\omega_0^*$.

Finally, feasibility follows by definition of $\Delta P$ since
\begin{align*}
\sum_{i=1}^n {\delta q_i^g}^*+{\delta q_i^r}^*&= -(\sum_{i=1}^n{R_i^g}^{-1}+\sum_{i\in DC\cup VI} {R_i^r}^{-1})\omega_0^*\\
&=\sum_{i=1}^n p_i^{in}+{q_i^r}^*-{D_i}\omega_0^*=:\Delta P
\end{align*}
where the second steps follows from \eqref{eq:omega_0}.
Thus (${\delta q_i^g}^*,{\delta q_i^r}^*$) is feasible and satisfies the KKT conditions. Therefore, (${\delta q_i^g}^*,{\delta q_i^r}^*$) is an optimal allocation. Since the optimal solution is unique, we must have
$({\delta q_i^g}^*,{\delta q_i^r}^*)=(\hat {\delta q_i^g}^*,\hat {\delta q_i^r}^*)$.
\end{proof}

\begin{remark}
Theorem \ref{thm:droop-control-optimality} illustrates the versatility of droop control and how it can be used to optimally accommodate supply-demand imbalances. However, it also highlights the need to tune parameters that have a direct effect on the network dynamics in order to achieve optimal steady state performance.
\end{remark}

\begin{remark}
The use of quadratic costs is standard in the power system literature~\cite{Kirschen:2004ja}. However, since this paper focuses mostly on local analysis, the analysis can be extended to include nonlinear costs by substituting $\alpha_i^{z}$ with $\frac{\partial}{\partial x}c_i^{z}(x)|_{x=\delta {q_i^{z}}^*}$ where $\delta {q_i^{z}}^*$ denotes the equilibrium value of $q_i^{z}$ ($z$ refers to either $g$ or $r$).
\end{remark}

\subsection{Dynamic Performance}\label{ssec:dyn-performance}
We now focus our attention on the dynamic performance metric. In general, there are several metrics that can be considered~\cite{Miller:2011tm}, and the change of droop control coefficients can have direct effect in all of them. Here we focus $\mathcal H_2$ norm of the system when the output is the frequency vector $\omega$ and the system is being driven by stochastic white noise.

The main motivation of this particular choice of metric is the need to characterize the effect of the increasing generation volatility as well as the effect of the measurement noise. We argue that while generator-based droop control can be usually implemented without major measurement noise, implementing droop control on inverters that are distributed all over the transmission and distribution network needs to account for these errors. Thus the use of $\mathcal H_2$-norm seems natural as it already accounts for stochastic inputs.
%In particular, we show that the virtual inertia operational mode generates unbounded norm, and motivates the need of an alternative solution. 

We assume that the net power injection of each bus is given by $P_i^{in} + k_i^1w_i^1(t)$, where $w_1(t)=(w_i^1(t),i\in N)$ is a vector of uncorrelated stochastic white noise with unit variance ($E[w_1^{T}(\tau)w_1(t)]=\delta(t-\tau)I_n$) that represents demand fluctuations.  
Similarly, we assume that for the purposes of implementing the droop control on the inverter the measured frequency is given by $\omega_i(t) + k_i^2w_i^2(t)$ where again $w_2(t)=(w_i^2(t),i\in N)$ is such that $E[w_2^{T}(\tau)w_2(t)]=\delta(t-\tau)I_n$. Finally, since to estimate $\dot\omega_i(t)$ one needs to obtain first $\omega_i$, we assume that the measurement value of the frequency derivative is  $\dot \omega_i + k_i^3w_i^3$, where $w_3=(w_i^3:=\frac{\text{d}}{\text{d}t}w_i^2,\,i\in N)$.\footnote{We use here the notation $\frac{\text{d}}{\text{d}t}w_i^2$ to represent the frequency weighted noise process with weight function given by $\text{w}_i(\bi 2\pi f)=\bi2\pi f$, see Remark \ref{rem:1} for more details. }

%
%Similarly, we replace the droop control term $\frac{1}{R_i^z}\omega_i$ ($z\in \{g,r\}$) by $\frac{1}{R_i^z}(\omega_i+k_i^2w_i^2)$ and the virtual inertia term $M^v_i\dot\omega_i$ by $M^v_i(\dot\omega_i+k_i^3w_i^3)$, where  $E[w^{2T}w^2]=E[w^{3T}w^3]=I_n$. 

Without loss of generality we make the following change of variables 
\begin{equation}\label{eq:change-variable}
\delta\theta(t) = \theta(t)-\omega_0^*t,\text{ and } \delta\omega(t) = \omega(t) - \omega_0^*.
\end{equation}
%We further drop the $\delta$ and interpret $\theta$ and $\omega$ as deviations from the steady state.
Thus, defining the system output to be $y(t)=\delta\omega(t)$ we can combine \eqref{eq:system} together with \eqref{eq:qr} to get the following MIMO system
\begin{subequations}\label{eq:mimo}
\begin{align}
\begin{bmatrix}
I \!\!&\!\! 0 \\
0\!\!&\!\!\hat M
\end{bmatrix}\!\!
\begin{bmatrix}
\dot {\delta\theta}\\
\dot  {\delta\omega}
\end{bmatrix}
\!\!&=\!\!
\begin{bmatrix}
0 \!\!&\!\! I\\
- L\!\!&\!\! -\hat D
\end{bmatrix}\!\!\!
\begin{bmatrix}
 \delta \theta\\
 \delta \omega
\end{bmatrix}
\!\!+\!\!
\begin{bmatrix}
0 \!\!&\!\! 0 \!\!&\!\!0\\
K_1 \!\!&\!\! -{R^r}^{-1}K_2    \!\!&\!\! -M^vK_3 
\end{bmatrix}\!\!\!
\begin{bmatrix}
w_1\\
w_2\\
w_3
\end{bmatrix}\\
y &= \begin{bmatrix} 0 &I \end{bmatrix}\begin{bmatrix}  \delta\theta \\  \delta\omega\end{bmatrix}
\end{align}
\end{subequations}
where $M^v=\diag(M^v_i)$, $K_x = \diag(k_i^x)$ and $w_x=(w_i^x)$, with $x\in\{1,\,2,\,3\}$,  $\hat M=M+M^v$, $\hat D=D+\hat R^{-1}$ and $\hat R^{-1}={R^{g}}^{-1}+{R^{r}}^{-1}$.

Notice that for simplicity the system \eqref{eq:mimo} implicitly assumes that all the inverters in the system are operated using the virtual inertia operation mode. However, it is possible to model using \eqref{eq:mimo} the droop controlled mode by setting $M^v_i=k_i^3=0$. Moreover, if one wants to model the constant power mode one just needs to additionally set $\frac{1}{R_i^r}=0$.
To simplify the discussion we will only consider two cases: 1) All the  inverters implement droop control ($DC=N$); 2) All the inverters implement the virtual inertia control ($VI=N$).

It will also be useful to write \eqref{eq:mimo} in standard form 
%\begin{subequations}
\begin{align}\label{eq:mimo-sf}
\begin{bmatrix}
\dot  {\delta\theta}\\
\dot  {\delta\omega}
\end{bmatrix}
\!\!&=\!\!
A
\begin{bmatrix}
  \delta\theta\\
  \delta\omega
\end{bmatrix}
\!\!+\!\!
B\begin{bmatrix}
w_1\\
w_2\\
w_3
\end{bmatrix},& 
y &= C\begin{bmatrix}  \delta\theta \\  \delta\omega\end{bmatrix},
\end{align}
%\end{subequations}
where
\begin{align}
A&=
\begin{bmatrix}
0 \!\!&\!\! I\\
- \hat M^{-1}L\!\!&\!\! -\hat M^{-1}\hat D
\end{bmatrix},\;\; C=\begin{bmatrix} 0 & I \end{bmatrix},\text{ and }\\
B &= \begin{bmatrix}
0 \!\!&\!\! 0 \!\!&\!\!0\\
\hat M^{-1}K_1 \!\!&\!\! -\hat M^{-1}{\hat R}^{-1}K_2    \!\!&\!\! -\hat M^{-1}M^vK_3 
%-\hat M^{-1}{\hat R}^{-1}K_1    \!\!&\!\!- \hat M^{-1}mK_2 \!\!&\!\!\hat M^{-1}K_3
\end{bmatrix}.
\end{align}

%Three of the performance metrics popularly used today~\addcite are associated to the response of the system's frequency to a step change on the power injection.  Interestingly, these three metrics have equivalents in the study of the step response of a second order system. For example, the Frequency Nadir (the maximum excursion of the frequency) is essentially the overshoot. 
%are the Frequency Nadir (maximum frequency excursion),  the time to reaand the final settling frequency. Interestingly, the latter is focuses on steady state characteristics while the former focuses on the overshoot that occurs after a big disturbance (which can be thought as a step response). 

Let $H$ denote the LTI system \eqref{eq:mimo}. Then the square of the  $\mathcal H_2$ norm of  \eqref{eq:mimo} can be formally defined using 
\begin{flalign}\label{eq:dyn-cost}
&\text{\bf Dyn-Cost:} &  ||H||_{\mathcal H_2}^2 &= \lim_{t\rightarrow\infty} E [y^T(t)y(t)]\qquad\qquad
\end{flalign}
where $y(t)$ is the output of \eqref{eq:mimo} when the input $w(t)=(w^1(t),w^2(t),w^3(t))$ is composed by a white noise process with unit covariance (i.e., $E[w_k(\tau)w_k^T(t)] = \delta(t -\tau)I$ for $k\in\{1,2\}$) and its derivative ($w_3=\frac{d}{dt}w_2$).

The computation of the $\mathcal H_2$ norm has been widely studied in modern control theory. In particular, in the case when $w_2$ and $w_3$ are not correlated processes ($w_3\not=\frac{d}{dt}w_2$), one very useful procedure to compute $||H||_{\mathcal H_2}$ (see ~\cite{Doyle:1989kf}) is based on using 
\begin{equation}\label{eq:H2}
||H||_{\mathcal H_2}^2=\tr (B^TXB)
\end{equation}
where $X$ is the observability Grammian, i.e. $X$ solves the Lyapunov equation 
\begin{equation}\label{eq:lyapunov}
A^TX+XA=-C^TC.
\end{equation}

In the context of power systems the use of this methodology has been first used in \cite{Tegling:2015ef}, where the authors seek to compute the power losses incurred by the network in the process of resynchronizing   generators after a disturbance.  Since then, several works have used similar metrics to evaluate effect of controllers on the power system performance, see e.g. \cite{Poolla:2015vq,Tegling:2016wna}.

%Here we leverage the tools developed in \cite{Tegling:2015ef} to motivate the need for a new mechanism to perform droop control.
\begin{remark}\label{rem:1}
It is important to notice that in our case the system is driven also by the derivative of the noise process $w_2$. Thus, \eqref{eq:H2} can only be applied when  $K_3=0$. When $K_3\not=0$, then \eqref{eq:dyn-cost} corresponds to a frequency weighted $\mathcal H_2$-norm. More precisely, then noise process $w_3$ is a frequency weighted process with weight function given by $W(s)=sI$, i.e. $\hat w_3(s) = s\hat w_2(s)$ where $\hat w_k(s)$ is the Laplace Transform of $w_k(t)$.
\end{remark}

The next theorem shows that droop controlled inverters can indeed affect the performance by changing droop parameters, and that inverters implementing virtual inertia can drastically degrade the performance when measurement noise is considered.
\begin{theorem}[$\mathcal H_2$-norm Computation]\label{thm:dyn-cost}
  Assume homogeneous parameter values, i.e. $M_i=m$, $M^v_i=m^v$, $D_i=d$, $R_i^g= r^g$, $R_i^r=r^r$, $k_i^1=k_1$, $k_i^2=k_2$ and $k_i^3=k_3$. Let 
  \begin{equation}\label{eq:notation}
  \hat r^{-1} = {r^g}^{-1}+{r^r}^{-1},\;\hat d= d+\hat r^{-1}, 
    \end{equation}
  and let $H_{DC}$ and $H_{VI}$ denote the MIMO system \eqref{eq:mimo} when all inverters implement droop control and virtual inertia, respectively.
  
 Then the squared $\mathcal H_2$ norm of $H_{DC}$ and $H_{VI}$ is given by 
\begin{align}\label{eq:h2-dc}
||H_{DC}||_{\mathcal H_2}^2 \!&=\!\frac{n\left((k_1)^2+(k_2{r^r}^{-1})^2 \right)}{2m(d+{r^g}^{-1}+{r^r}^{-1}) },
\intertext{ and }
||H_{VI}||_{\mathcal H_2}^2 \!&=\!+\infty,\label{eq:h2-vi}
\end{align}
respectively.
\end{theorem}
%\begin{proof}
%See Appendix.
%\end{proof}
\begin{proof}
The proof of this theorem is analogous to \cite[Lemmas 3.1 and 3.2]{Tegling:2015ef}. We study the two cases \eqref{eq:h2-dc} and \eqref{eq:h2-vi} separately.

\noindent
{\bf Computing $||H_{DC}||_{\mathcal H_2}^2$:} 
Notice that in this case $M^v=0$. In order to compute $||H_{DC}||_{\mathcal H_2}^2$ we first make the same change of variable used in \cite{Tegling:2015ef},
\[
 \delta\theta = U\theta' \text{ and } \delta \omega = U\omega',
\]
where $U$ is the orthonormal transformation that diagonalizes $L$, i.e. $U^TL_BU= \Gamma$ where $\Gamma=\diag\{\lambda_1=0,\dots,\lambda_n\}$ and $U$ is assumed w.l.o.g. to be of the form $U=[\frac{1}{\sqrt{n}}\mathbf 1_n \,|\,U_{n-1}]$.

Thus, if we further transform $y=Uy'$, $w_1=Uw_1'$ and $w_2=Uw_2'$, 
we can decouple \eqref{eq:mimo-sf} into $n$ subsystems given by 
\begin{flalign}\label{eq:decoupled-mimo}
&\text{${H_{DC,i}}$:}\!\!\!\! & \!\!\!\!\!\!
\begin{matrix}
\begin{array}{cl}
\begin{bmatrix}
\dot \theta_i'\\
\dot \omega_i'
\end{bmatrix}
\!\!&=\!\!
\bbm 0 \!&\! 1\\ -\frac{\lambda_i}{m} \!&\! -\frac{\hat d}{m} \ebm
\begin{bmatrix}
 \theta'\\
 \omega'
\end{bmatrix}
\!\!+\!\!
\bbm
0 \!&\! 0\\
\frac{k_1}{m} \!&\!-\frac{k_2}{mr^r}
\ebm
\begin{bmatrix}
{w^1_i}^{'}\\
{w^2_i}^{'}
\end{bmatrix} \\ 
y_i' &= \bbm 0 &1\ebm \begin{bmatrix} \theta_i' \\ \omega_i'\end{bmatrix}
\end{array},
\end{matrix}
\end{flalign}
where a simple computation using \eqref{eq:H2} and \eqref{eq:lyapunov} shows that 
$$||H_{DC,i}||_{\mathcal H_2}^2=\frac{\left((k_1)^2+(k_2{r^r}^{-1})^2 \right)}{2m(d+{r^g}^{-1}+{r^r}^{-1}) }.
$$
Therefore, since $||H_{DC}||_{\mathcal H_2}^2 = \sum_{i=1}^n||H_{DC,i}||_{\mathcal H_2}^2$ we obtain \eqref{eq:h2-dc}.

\noindent
{\bf Computing $||H_{VI}||_{\mathcal H_2}^2$:} 
To show that the norm $||H_{VI}||_{\mathcal H_2}^2=+\infty$ we will show that the transfer function of $H_{VI}$ has nonzero feedthrough ($\sigma_{\max}(H(\bi\infty))\geq \varepsilon>0$).

To compute the transfer function we first notice that $w_3=\dot w_2$ which implies that we can model in the Laplace domain $\hat w_3(s) = s\hat w_2(s)$. 

Similar to the DC case, we can use a change of variable to decouple the system into $n$ different modes given by
\begin{align*}
\begin{bmatrix}
\dot \theta_i'\\
\dot \omega_i'
\end{bmatrix}
\!\!&=\!\!
A_i\begin{bmatrix}
 \theta'\\
 \omega'
\end{bmatrix}
\!\!+\!\!
B_i\begin{bmatrix}
{w^1_i}^{'}\\
{w^2_i}^{'}\\
{w^3_i}^{'}
\end{bmatrix},\;&
y_i' &= C_i \begin{bmatrix} \theta_i' \\ \omega_i'\end{bmatrix}
\end{align*}\vspace{-.1in}
\begin{flalign*}
&\text{ with }&&A_i=\bbm 0 \!&\! 1\\ -\frac{\lambda_i}{\hat m} \!&\! -\frac{\hat d}{\hat m} \ebm,\;
B_i=\bbm
0 \!&\! 0\!&\! 0\\
\frac{k_1}{\hat m} \!&\!-\frac{k_2}{\hat mr^r}&\!-\frac{k_3m^v}{\hat m}
\ebm\\
&&&C_i=\bbm 0 & 1\ebm
\end{flalign*}
where $\hat m =m+m^v$.

We drop the subscript $i$ and define $\hat w'(s) = [\hat w'_1(s)\,\hat w'_2(s)\,\hat w'_3(s)]^T$. Thus we  can compute the transfer function $H(s)$ using
\begin{align*}
&\hat y(s)=H(s)w'(s)= C(sI-A)^{-1}B \hat w'(s)\\
&=\frac{1}{\Delta(s)}C\bbm s +\frac{\hat d}{\hat m}\!&\! 1\\ -\frac{\lambda}{\hat m} \!&\! s\ebm B \hat w'(s)\\
&=\frac{1}{\Delta(s)}\bbm  -\frac{\lambda}{\hat m}\!&\! s\ebm B \hat w'(s)\\
&=\frac{s}{\Delta(s)} 
\bbm
\frac{k_1}{\hat m} \!&\!-\frac{k_2}{\hat mr^r}&\!-\frac{k_3m^v}{\hat m}
\ebm \hat w'(s)\\
&=\frac{s}{\Delta(s)} 
\bbm
\frac{k_1}{\hat m} \!&\!-\frac{k_2}{\hat mr^r}&\!-\frac{k_3m^v}{\hat m}
\ebm  \bbm \hat w_1'(s) \\ \hat w_2'(s)\\\hat w_3'(s)\ebm\\
&=\frac{s}{s^2+\frac{\hat d}{\hat m}s + \frac{\lambda}{\hat m}} 
\bbm
\frac{k_1}{\hat m} \!&\!-\left(s\frac{k_3m^v}{\hat m}+ \frac{k_2}{\hat mr^r}\right)
\ebm  \bbm \hat w_1'(s) \\ \hat w_2'(s)\ebm
\end{align*}
where $\Delta(s)=s^2+\frac{\hat d}{\hat m}s + \frac{\lambda}{\hat m}$ and in the last step we used the relationship $\hat w_3'(s) = s\hat w_2'(s)$.
It follows that when $f \rightarrow +\infty$,
\[
H(\bi 2\pi f)\rightarrow \bbm 0& -\frac{k_3m^v}{\hat m}\ebm,
\]
which implies that $||H_{VI}||_{\mathcal H_2}^2=+\infty$.
\end{proof}

Theorem \ref{thm:dyn-cost}  provides an interesting insight on how the different controllers described in Section \ref{sec:model} affect the system performance and illustrates the effect of measurement errors on it. To understand the performance changes, we provide the $\mathcal H_2$ norm of the swing equations ($H_{SW}$) without any additional control ($r^r=m^v=k_2=k_3=0$)
\begin{equation}
||H_{SW}||_{\mathcal H_2}^2 = \frac{n}{2m(d+{r^g}^{-1})}(k_1)^2.
\end{equation}
Thus it is easy to see that adding droop control introduces a larger damping $(d+{r^g}^{-1}+{r^r}^{-1})>d+{r^g}^{-1}$. However, there is an intrinsic tradeoff since if $r^r$ is small enough, then the noise amplification takes over and increases the norm. %\enrique{if possible add a figure}
But perhaps the most interesting result from this theorem is \eqref{eq:h2-vi}, which implies that the massive use of virtual inertia in the network can amplify the stochasticity of the system and introduce huge volatility in the frequency fluctuations.

\subsection{The Need for a Better Solution}

The analysis provided in sections \ref{ssec:ss-performance} and \ref{ssec:dyn-performance} shows that none of the existing solutions can simultaneously achieve efficient steady state operation while improving dynamic performance. 

On the one hand, while Theorem {\ref{thm:droop-control-optimality} shows indeed that droop control can be used to optimally allocate resources by carefully setting the parameters $R^r_i=\alpha_i^r$ and $R^g_i=\alpha_i^g$, this tie between control parameters and economic efficiency makes it impossible to further improve the dynamic performance without incurring on an additional steady state cost \eqref{eq:ss-cost}. Therefore, if one wants to operate the system in an efficient steady state, then $R^r_i$ cannot be used to improve the dynamic performance.

On the other hand, while the use of virtual inertia has been suggested to be a viable solution to improve the dynamic performance without losing steady state efficiency, Theorem \ref{thm:dyn-cost} shows that this solution cannot be widely adopted since it will induce a large noise amplification that can hinder the secure operation of the power grid.

As a result, all the existing solutions for inverter control  cannot provide a dynamic performance improvement without sacrificing either steady-state  or dynamic performance.

\section{Dynam-{i}-c {Droop} Control (iDroop)}\label{sec:idroop}
We now introduce our iDroop control. The main underlying idea on the design of the proposed controller is to leverage the flexibility that virtual inertia controllers provide while controlling the noise amplification using a filtering stage. 

Let $ID\subset N$ be the set of network buses that implement iDroop. Then we propose to control the  power injected to the grid by the inverter using 
\begin{flalign}\label{eq:idroop}
&\text{{\bf iDroop:}}&
\begin{matrix}
\begin{array}{rl}
q_i \!\!&= q_i^{r,0} + x_i\\
\dot x_i \!\!&=  \delta_i\left(-\frac{1}{R^r_i}\omega_i - x_i \right)-\nu_i\dot\omega_i
\end{array}
\end{matrix}\qquad\qquad
\end{flalign}

A few comments are in order. Firstly, the parameter $\nu_i$ place a role similar to $M^v_i$ in the virtual inertia controller. However, since this term is integrated in \eqref{eq:idroop}, it is not longer interpreted as a virtual inertia. Secondly, the parameters $\delta_i$ and $\nu_i$ can be  independently tuned to reduce the  noise introduced by the frequency ($k_i^2w^2_i$) or the  frequency derivative ($k_i^3w^3_i=k_i^3\dot w^2_i$).
This can be easily seen when these noise processes are introduced in \eqref{eq:idroop} giving
\begin{equation}
\dot x_i = \delta_i\left(-\frac{1}{R^r_i}\omega_i - x_i \right)-\nu_i\dot\omega_i - \frac{\delta_ik_i^2}{R_i^r}w_i^2 - \nu_ik_i^3\dot w_i^2.\label{eq:idroop-b-noise}
\end{equation}
Finally, while the stability (in the absence of noise) of \eqref{eq:system} with \eqref{eq:qr} is trivially guaranteed, the additional integration stage requires a more detailed analysis.

We do not provide here an explicit formula for $||H_{iDroop}||_{\mathcal H_2}$ in this paper and leave it for future research. Instead we will numerically illustrate in the next section the effect of $\nu_i$ and $\delta_i$ on this norm and how the performance of iDroop compares with the one of droop controlled inverters.
In the rest of this section we show that indeed our controllers are able preserve the steady state behavior for arbitrary parameter values $\delta_i$ and $\nu_i$, and characterize a sufficient condition for asymptotic stability.

\subsection{Steady State Optimality}
We now show that our iDroop controllers  provide the same steady state properties as traditional droop control. We do this by showing that iDroop achieves the minimum  {\bf SS-Cost} \eqref{eq:ss-cost}.

\begin{theorem}[iDroop Optimality]\label{thm:idroop-optimality}
Let $(\theta^*(t),\allowbreak\omega^*\!\!,\allowbreak{q^r}^*)$ be the steady state solution of \eqref{eq:system} where $q_i^r$ is given by \eqref{eq:idroop}. Then the steady state solution of \eqref{eq:system} and \eqref{eq:idroop} is given by 
$$ \theta(t)^* \!\!= \theta^* + \mathbf 1_n\omega_0^*t,\,\, \omega_i^* = \omega_0^*\,\,\forall i \text{ and }q_i^*-q_i^{r,0}= x_i^*=-\frac{\omega_0^*}{R_i^r}{}.$$
Moreover, if $R_i^r=\alpha_i^r$ and $R_i^g=\alpha_i^g$, then $\delta {q_i^g}^* := -\frac{1}{R_i^g}\omega_0^*$, $\delta {q_i^r}^* := {q_i^r}^*-q_i^{r,0}$ is the unique allocation  that minimizes the steady state cost \eqref{eq:ss-cost} subject to \eqref{eq:supply-demand-balance}, where $\Delta P=\sum_{i=1}^n p_i^{in}+q_i^{r,0} - D_i\omega_0^*$.
\end{theorem}
\begin{proof}
The proof of this theorem relies on Theorem \ref{thm:droop-control-optimality}. We will show that iDroop has a steady state behavior that is identical to the standard droop control when $DC\cup VI=ID$. Thus, we can use then use Theorem \ref{thm:droop-control-optimality} to show that indeed iDroop preserves the optimality characteristics of the traditional droop control.

We now characterize the steady state behavior of \eqref{eq:system} and \eqref{eq:idroop}. 
Similarly to  \eqref{eq:system} and \eqref{eq:qr}, we can use \eqref{eq:system-b} to show that $\dot \omega=0$ only if $\theta(t)=\theta^*(t) := \theta^*+\mathbf 1_n\omega_0^*t$.  Using \eqref{eq:idroop} with $\dot x=0$ and $\dot \omega=0$, we get that $R_i^rx_i^* = -\omega_i^*=-\omega_0^*$. Therefore, the steady state deviation of the inverter $i$ is given by $$\delta {q_i^r}^*=q_i^*-q_i^{r,0} = x_i^*=-\frac{1}{R_i^r}\omega_0^*.$$

Since the steady state deviation of conventional generators is by definition ${\delta q_i^g}^*=-\frac{1}{R_i^g}\omega_0^*$, then we have obtained the same allocation described in Theorem \ref{thm:droop-control-optimality}. It follows then that when $\alpha_i^r=R_i^r$ and $\alpha_i^g=R_i^g$, the steady state solution of \eqref{eq:idroop} is an allocation that minimizes \eqref{eq:ss-cost} subject to \eqref{eq:supply-demand-balance}.
\end{proof}

\subsection{Stability Analysis}

We now show that the controllers described in \eqref{eq:idroop} can preserve the stability of the network. 

We use the same change of variable used in \eqref{eq:change-variable} together with $ \delta x(t) = x(t) - x^*$.
%Without loss of generality we make the following change of variables 
%$$ \delta\theta(t) = \theta(t)-\omega_0^*t,\; \delta\omega(t) = \omega(t) - \omega_0^*,\text{ and } \delta x(t) = x(t) - x^*.$$
%To simplify notation in this section we drop the $\delta$ from the variable and interpret $\theta$, $\omega$ and $x$ as deviations from the steady state.
%
Therefore \eqref{eq:system} and \eqref{eq:idroop} become
\begin{subequations}\label{eq:system-idroop}
\begin{align}
\dot{ \delta \theta} &=  \delta\omega\label{eq:system-idroop-a}\\
M\dot{  \delta\omega} & = - (D + {R^g}^{-1}) \delta\omega -L_B  \delta\theta + \delta x\label{eq:system-idroop-b}\\
\dot{ \delta x}&= -K_\delta({R^r}^{-1} \delta\omega +   \delta x) -K_\nu\dot{ \delta\omega}\label{eq:system-idroop-c}
\end{align}
\end{subequations}
where $K_\nu=\diag(\nu_i,i\in N)$.
\footnote{To simplify notation we assume here that every bus iDroop. However, the results can be generalized for any combination of iDroop with the controllers in \eqref{eq:qr}.}
It is easy to see that now the steady state solutions of Theorem \ref{thm:idroop-optimality} become
$( \delta\theta^*, \delta\omega^*, \delta x^*)=(\mathbf 1_n\alpha,0,0)$ for any $\alpha\in\mathds R$. Thus the set of equilibria of \eqref{eq:system-idroop} is given by
\begin{equation}\label{eq:equilibria}
\hat {\mathcal E}=\{( \delta\theta, \delta\omega, \delta x):  \delta\omega= \delta x=0,\, \delta\theta=\mathbf 1_n\alpha,\,\alpha\in\mathds R\}.
\end{equation}

\begin{theorem}[Asymptotic Convergence]
Whenever the following condition holds,
%\begin{equation}\label{eq:condition}
%\nu_i \in \left( -\frac{1}{R_i^r} - 2\sqrt{D_i+\frac{1}{R_i^g}},-\frac{1}{R_i^r} + 2\sqrt{D_i+\frac{1}{R_i^g}}\right)
%\end{equation}
\begin{equation}\label{eq:condition}
\frac{\nu_i}{\delta_i(\nu_i\!+\!{R_i^r}^{-1})}>0\;\text{ and }\;(D_i\!+\!{R_i^g}^{-1})+\frac{\nu_i{R_i^r}^{-1}}{\nu_i+{R_i^r}^{-1}}>0,
%\nu_i \in \left( -\frac{1}{R_i^r} - 2\sqrt{D_i+\frac{1}{R_i^g}},-\frac{1}{R_i^r} + 2\sqrt{D_i+\frac{1}{R_i^g}}\right)
\end{equation}
the iDroop control \eqref{eq:idroop} converges asymptotically to the set of equilibria $\hat {\mathcal E}$ described in \eqref{eq:equilibria}.
\end{theorem}
\begin{proof}
We will show that under the conditions of the theorem, the following Lyapunov function decreases along trajectories:
\begin{equation}\label{eq:V}
\begin{aligned}
V( \delta\theta, \delta\omega, \delta x)=& \frac{1}{2} \delta \theta^T\!L_B \delta\theta+\frac{1}{2} \delta\omega^TM \delta\omega \\&+ \frac{1}{2}( \delta x+K_\nu \delta\omega)^TT( \delta x+K_\nu \delta\omega).
\end{aligned}
\end{equation}
where $T\in \R^{n\times n}$ is a positive definite diagonal matrix to be defined later ($T\succ 0$ and $T=\diag(v)$, with $v\in \R^n$).
The Lyapunov function $V$ is inspired on \cite{Paganini:2009fk} where a similar derivative term is used to damp oscillations.

The change of $V$ along the trajectories of \eqref{eq:system-idroop} is given by
\begin{align}
\dot V =&\, \delta\theta^TL_B\dot{ \delta\theta}\!+\! \delta\omega^TM\dot{ \delta \omega}\!+\! (\delta x\!+\!K_\nu\delta \omega)^TT(\dot {\delta x}\!+\!K_\nu\dot{\delta  \omega})\label{eq:dotv-1}\\
=&\,\delta \theta^TL_B\delta \omega+ \delta \omega^T(-(D + {R^g}^{-1})\delta \omega - L_B\delta \theta +\delta x) \nonumber\\
&+(\delta x+K_\nu\delta \omega)^TTK_\delta(-{R^r}^{-1}\delta \omega -\delta x)\label{eq:dotv-2}\\
=&\,-\delta \omega^T(D \!+\! {R^g}^{-1} \!+\! K_\nu TK_\delta {R^r}^{-1})\,\delta \omega \!-\!\delta x^TTK_\nu \delta x\nonumber\\
&+\delta \omega^T(I-K_\nu TK_\delta\!-\!TK_\delta{R^r}^{-1})\delta x\label{eq:dotv-3}
%=&\,-\omega^T(D + {R^g}^{-1})\omega -||x||^2 -x^T({R^r}^{-1}+K_\nu)\omega\label{eq:dotv-4}\\
%=&\,\bbm \omega^T \!\!&\!\! x^T\ebm 
%\bbm 
%- (D+{R^g}^{-1}) \!\!&\!\! \frac{-1}{2}({R^r}^{-1} \!+\!K_\nu)\\
%\frac{-1}{2}({R^r}^{-1} \!+\!K_\nu) \!\!&\!\! -I
%\ebm
%\bbm \omega\\ x\ebm 
%\label{eq:dotv-4}\\
%=&\!:\,\bbm \omega^T \!\!&\!\! x^T\ebm Q  \bbm \omega\\ x\ebm \label{eq:dotv-5}
\end{align}
where \eqref{eq:dotv-2} follows from \eqref{eq:system-idroop}, and \eqref{eq:dotv-3} from rearranging the terms.

We now choose $T$ so that the cross term in \eqref{eq:dotv-3} is zero. Since $T$ is assumed to be a diagonal matrix, then we get
\begin{align}
0&=I-K_\nu TK_\delta\!-\!TK_\delta{R^r}^{-1}  &\iff\nonumber\\
0&=I-TK_\delta(K_\nu + {R^r}^{-1})& \iff\nonumber\\
T& = {K_\delta}^{-1} (K_\nu + {R^r}^{-1})^{-1}\label{eq:T}
\end{align}

Therefore, using \eqref{eq:T}, it follows that 
\begin{align}
\dot V=& -\delta \omega^T(D \!+\! {R^g}^{-1} \!\!+\! K_\nu {R^r}^{-1}(K_\nu \!+\! {R^r}^{-1})^{-1} )\delta \omega \nonumber\\
&-\delta x^TK_\nu K_\delta^{-1}(K_\nu \!+\! {R^r}^{-1})^{-1} \delta x\label{eq:dotv-4}
\end{align}

Thus, it follows from \eqref{eq:condition} that $\dot V\leq0$ and we can now apply LaSalle's Invariance Principle\cite{khalil2002nonlinear} to show that $(\delta \theta(t),\delta \omega(t),\delta x(t))$ converges to the largest invariant set $M\subset\{\dot V\equiv 0\}$.  Using  \eqref{eq:dotv-4} we can show that  $\dot V\equiv 0$ implies that $\delta \omega(t)\equiv 0$ and $\delta x(t)\equiv 0$, which implies in turn (through \eqref{eq:system-idroop-a}) that $\dot {\delta \theta}\equiv 0$. Therefore, it must follow that $\delta \theta(t)\equiv \delta \theta^*$. Finally, using \eqref{eq:system-idroop-b} we get $0\equiv L_B\delta \theta(t)\equiv L_B\delta \theta^*$ which, since the graph is connected, implies that $\theta^*=\mathbf 1_n\alpha$ for some $\alpha\in\mathds R$. Therefore, every trajectory converges to the set of equilibria $\hat{\mathcal E}$ given in \eqref{eq:equilibria}, i.e. $M\subset \hat{\mathcal E}$.
\end{proof}

\begin{figure}[htp]
\includegraphics[width=\columnwidth]{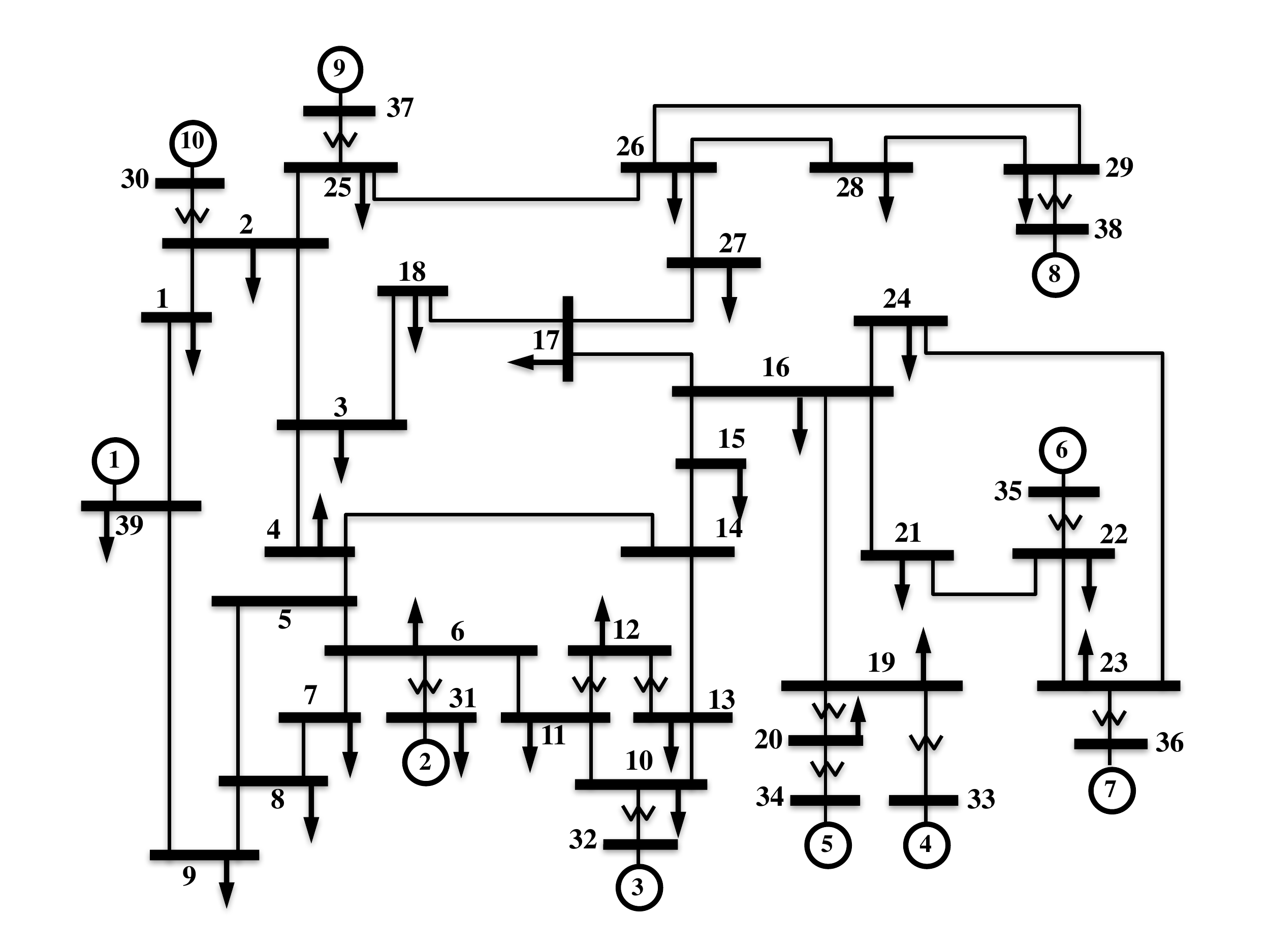}
\caption{IEEE 39 Bus System (New England)}\label{fig:new-england}
\end{figure}

\section{Numerical Illustrations}\label{sec:simulations}

%While we do not provide a analytic expression for $||H_{\text{iDroop}}||_{\mathcal H_w}^2$,  
In this section we numerically illustrate some of the features of our iDroop control using  the IEEE 39 bus system shown in Fig. \ref{fig:new-england}. The network parameters as well as the stationary starting point were obtained from the Power System Toolbox (PST)\cite{chow1992toolbox} dataset. 

Before building the dynamic model, we perform the Kron reduction for every load bus. Thus the final system has 10 buses that correspond to each generator in the nework. The inertia parameters of each generators are also obtained from the PST dataset. Throughout the time domain simulations we assume that the aggregate generator damping and load frequency sensitivity parameter is $D_i=0.1$. We also set the droop coefficient of each generator to $R_i^g=15$.

On each bus we add an additional inverter-based generator and vary their operational mode using either one of the three modes described in \eqref{eq:qr} (CP=constant power, DC=droop control and VI=virtual inertia), or the iDroop mode. The droop coefficient is set to $R_i^r=15$ in all the cases. For illustration purposes we select parameters so that the VI mode and iDroop have a similar frequency transient behavior. In particular, we choose $M_i^v=0.15$, $\delta_i=6$, $\nu_i=0.9$.

\begin{figure}[htp]
\includegraphics[width=\columnwidth]{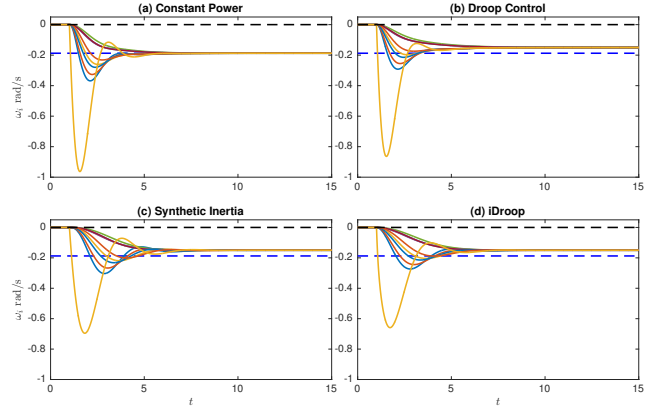}
\caption{Frequency deviations after perturbation}\label{fig:frequency}
\end{figure}

We initialize the system in steady state and at time $t=1$ we introduce a step change $\Delta P_{30}=-0.5$p.u. in the power injection at bus 30 (where generator 10 is located). 
Fig. \ref{fig:frequency} shows the evolution of the frequency deviations for the four cases: (a) CP, (b) DC, (c) VI and (d) iDroop. It can be seen that both VI (Fig. \ref{fig:frequency}-c) and iDroop (Fig. \ref{fig:frequency}-d) can reduce the Nadir (minimum frequency achieved), while achieving the the same steady state as the DC mode (Fig. \ref{fig:frequency}-b).

\begin{figure}[htp]
\includegraphics[width=\columnwidth]{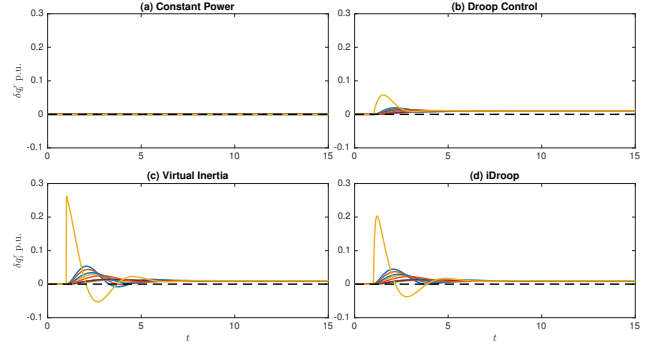}
\caption{Power deviation of inverter}\label{fig:power}
\end{figure}

However, the way this is achieved is in the case of VI and iDroop is completely different. This can be seen in Fig. \ref{fig:power} where we show the power deviation experienced by each individual inverter. In particular, the step change in the power injection introduces a discontinuity in $\dot\omega_{30}$. This is shown in Fig. \ref{fig:power}-c where the inverter at bus 30 has a step change in power. iDroop on the other hand is able to perform a similar task using less peak power and with a more desirable behavior.
This drastic change on the power experienced by the VI also illustrates why it is expected that this solution will produce noise amplification. 
%since instantaneous changes in power injection imply instantaneous changes in inverters input power.

%To make the comparison fair we make the following change of variable $\nu_i=\delta_iM_i^v$ in \eqref{eq:idroop} which makes  
%\[
%\dot x_i =\gamma_i \left(\delta_i(-\frac{1}{R_i^r}\omega -x_i)  -M_i^v\dot\omega_i\right).
%\]
%This new notation shows the added flexibility that the proposed mechanism provides as one can still tune $\delta_i$ even after having $R_i^r$ and $M_i^v$ fixed. 
\begin{figure}[htp]
\includegraphics[width=\columnwidth]{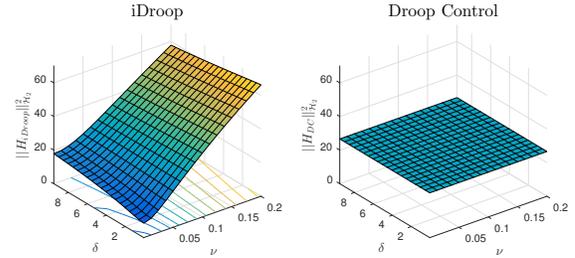}
\caption{$\mathcal H_2$-norm comparison between iDroop and Droop Control}\label{fig:h2}
\end{figure}

\begin{figure}[htp]
\includegraphics[width=\columnwidth]{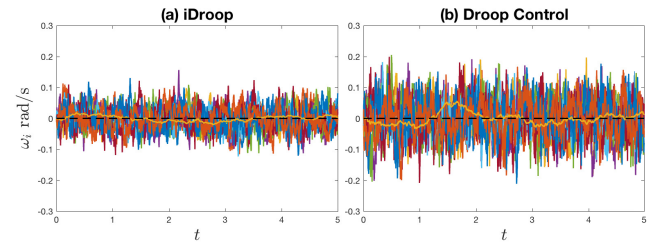}
\caption{Frequency fluctuations when $\delta_i=6$ and $\nu_i=0.01$}\label{fig:noise-sim}
\end{figure}

Finally, since the VI mode has unbounded $\mathcal H_2$-norm, we show in Fig. \ref{fig:h2} the   $\mathcal H_2$-norm of the system when the inverters are using only the DC mode and iDroop on a high measurement noise regime ($k_i^1=0.1$, $k_i^2=k_i^3=5$). We can see that DC does not vary with $\delta$ and $\nu$ (as expected) while iDroop does. However, more importantly, iDroop is able to achieve better performance than DC despite using frequency derivative measurements. This is also illustrated in Figure \ref{fig:noise-sim} where we simulate the frequency fluctuations  for the values of $\delta_i=6$ and $\nu_i=0.01$.

\section{Concluding Remarks}\label{sec:conclusions}
This paper studies the intrinsic trade off between steady state and dynamic performance of inverter-based droop control and several of its variants. We show that while the standard droop control can improve the dynamic performance, it can only achieve it by losing steady state efficiency. Moreover, our analysis also shows that the popular alternative of adding virtual inertia is very sensitive to measurement noise and can increase the frequency variance unboundedly.

To solve these issues we propose a new control scheme (iDroop) that is able to tune dynamic performance without altering steady state efficiency. We characterize the set of parameter values that can guarantee the stability of the steady state solution and illustrate its behavior numerically. In particular, we show that iDroop is able to reduce the Nadir and variance when compared with Constant Power (CP) and Droop Control (DC) modes.

\bibliographystyle{IEEEtran}
\bibliography{biblio.bib}

\begin{thebibliography}{10}
\providecommand{\url}[1]{#1}
\csname url@samestyle\endcsname
\providecommand{\newblock}{\relax}
\providecommand{\bibinfo}[2]{#2}
\providecommand{\BIBentrySTDinterwordspacing}{\spaceskip=0pt\relax}
\providecommand{\BIBentryALTinterwordstretchfactor}{4}
\providecommand{\BIBentryALTinterwordspacing}{\spaceskip=\fontdimen2\font plus
\BIBentryALTinterwordstretchfactor\fontdimen3\font minus
  \fontdimen4\font\relax}
\providecommand{\BIBforeignlanguage}[2]{{%
\expandafter\ifx\csname l@#1\endcsname\relax
\typeout{** WARNING: IEEEtran.bst: No hyphenation pattern has been}%
\typeout{** loaded for the language `#1'. Using the pattern for}%
\typeout{** the default language instead.}%
\else
\language=\csname l@#1\endcsname
\fi
#2}}
\providecommand{\BIBdecl}{\relax}
\BIBdecl

\bibitem{Ziebolz:1947tq}
H.~Ziebolz and A.~R. Co, ``{Control means for power generating systems},''
  Patent, Dec., 1947.

\bibitem{Almeras:1951vh}
P.~Almeras and N.~B.~P. Picte, ``{Regulation of an assembly of electric
  generating units},'' Jul. 1951.

\bibitem{Machowski:2011uc}
J.~Machowski, J.~Bialek, and D.~J. Bumby, \emph{{Power System Dynamics}}, ser.
  Stability and Control.\hskip 1em plus 0.5em minus 0.4em\relax John Wiley {\&}
  Sons, Aug. 2011.

\bibitem{WoodWollenberg1996}
A.~J. Wood, B.~F. Wollenberg, and G.~B. Sheble, ``{Power generation, operation
  and control},'' \emph{John Wiley{\&}Sons}, 1996.

\bibitem{Kirby:2005uy}
B.~J. Kirby, ``{Frequency Regulation Basics and Trends},'' Oak Ridge National
  Laboratory, Tech. Rep., Jan. 2005.

\bibitem{Zhao:2014bp}
C.~Zhao, U.~Topcu, N.~Li, and S.~H. Low, ``{Design and Stability of Load-Side
  Primary Frequency Control in Power Systems},'' \emph{Automatic Control, IEEE
  Transactions on}, vol.~59, no.~5, pp. 1177--1189, 2014.

\bibitem{Zhang:2015kq}
X.~Zhang and A.~Papachristodoulou, ``{A real-time control framework for smart
  power networks: Design methodology and stability},'' \emph{Automatica},
  vol.~58, pp. 43--50, Aug. 2015.

\bibitem{Dorfler:2014uu}
F.~D{\"o}rfler, J.~Simpson-Porco, and F.~Bullo, ``{Breaking the Hierarchy:
  Distributed Control {\&} Economic Optimality in Microgrids},''
  \emph{arXiv.org}, Jan. 2014.

\bibitem{Li:2014we}
N.~Li, L.~Chen, C.~Zhao, and S.~H. Low, ``{Connecting automatic generation
  control and economic dispatch from an optimization view},'' in
  \emph{Proceedings of the American Control Conference}, Massachusetts Inst. of
  Technology, Cambridge, United States.\hskip 1em plus 0.5em minus 0.4em\relax
  IEEE, Jan. 2014, pp. 735--740.

\bibitem{Mallada:2014ui}
E.~Mallada, C.~Zhao, and S.~H. Low, ``{Optimal load-side control for frequency
  regulation in smart grids},'' \emph{arXiv.org}, Oct. 2014.

\bibitem{mallada2014distributed}
E.~Mallada and S.~H. Low, ``{Distributed frequency-preserving optimal load
  control},'' in \emph{IFAC World Congress}, 2014, pp. 5411--5418.

\bibitem{Motter:2013iw}
A.~E. Motter, S.~A. Myers, M.~Anghel, and T.~Nishikawa, ``{Spontaneous
  synchrony in power-grid networks},'' \emph{Nature Physics}, vol.~9, no.~3,
  pp. 191--197, Feb. 2013.

\bibitem{Shen:1954eo}
D.~W.~C. Shen and J.~S. Packer, ``{Analysis of Hunting Phenomena in Power
  Systems by Means of Electrical Analogues},'' \emph{Proceedings of the IEE -
  Part II: Power Engineering}, vol. 101, no.~79, pp. 21--34, Feb. 1954.

\bibitem{Bergen:1981hja}
A.~R. Bergen and D.~J. Hill, ``{A Structure Preserving Model for Power System
  Stability Analysis},'' \emph{Power Apparatus and Systems, IEEE Transactions
  on}, vol. PAS-100, no.~1, pp. 25--35, 1981.

\bibitem{Lee:2013fk}
C.~T. Lee, R.~P. Jiang, and P.~T. Cheng, ``{A grid synchronization method for
  droop-controlled distributed energy resource converters},'' \emph{Ieee
  Transactions on Industry Applications}, vol.~49, no.~2, pp. 954--962, 2013.

\bibitem{CHANDORKAR:1993tp}
M.~C. Chandorkar, D.~M. Divan, and R.~Adapa, ``{Control of Parallel Connected
  Inverters in Standalone Ac Supply-Systems},'' \emph{Ieee Transactions on
  Industry Applications}, vol.~29, no.~1, pp. 136--143, 1993.

\bibitem{SimpsonPorco:2013gs}
J.~W. Simpson-Porco, F.~D{\"o}rfler, and F.~Bulbo, ``{Synchronization and power
  sharing for droop-controlled inverters in islanded microgrids},''
  \emph{Automatica}, vol.~49, no.~9, pp. 2603--2611, Sep. 2013.

\bibitem{Liu:2016kr}
J.~Liu, Y.~Miura, and T.~Ise, ``{Comparison of Dynamic Characteristics Between
  Virtual Synchronous Generator and Droop Control in Inverter-Based Distributed
  Generators},'' \emph{Power Electronics, IEEE Transactions on}, vol.~31,
  no.~5, pp. 3600--3611, May 2016.

\bibitem{Beck:eb}
H.~P. Beck and R.~Hesse, ``{Virtual synchronous machine},'' in \emph{2007 9th
  International Conference on Electrical Power Quality and Utilisation,
  EPQU}.\hskip 1em plus 0.5em minus 0.4em\relax TU Clausthal,
  Clausthal-Zellerfeld, Germany, Dec. 2007.

\bibitem{Driesen:ft}
J.~Driesen and K.~Visscher, ``{Virtual synchronous generators},'' in \emph{IEEE
  Power and Energy Society 2008 General Meeting: Conversion and Delivery of
  Electrical Energy in the 21st Century, PES}.\hskip 1em plus 0.5em minus
  0.4em\relax IEEE, New York, United States, Sep. 2008.

\bibitem{Varaiya:ig}
P.~Varaiya, F.~F. Wu, and R.-L. Chen, ``{Direct Methods for Transient Stability
  Analysis of Power Systems: Recent Results},'' \emph{Proceedings of the Ieee},
  vol.~73, no.~12, pp. 1703--1715, Jan. 1985.

\bibitem{Quazza:1970cd}
G.~Quazza, ``{Automatic control in electric power systems},''
  \emph{Automatica}, vol.~6, no.~1, pp. 123--150, Jan. 1970.

\bibitem{Quazza:1966vx}
------, \emph{{Criteria for equitable participation of areas in tie-line power
  and frequency control of an interconnected power system}}.\hskip 1em plus
  0.5em minus 0.4em\relax Automazione e Strumentazione, 1966.

\bibitem{deMello:1973jy}
F.~deMello, R.~Mills, and W.~B'Rells, ``{Automatic Generation Control Part
  II-Digital Control Techniques},'' \emph{Power Apparatus and Systems, IEEE
  Transactions on}, vol. PAS-92, no.~2, pp. 716--724, 1973.

\bibitem{Zhao:2015kk}
C.~Zhao, E.~Mallada, and S.~H. Low, ``{Distributed generator and load-side
  secondary frequency control in power networks},'' \emph{2015 49th Annual
  Conference on Information Sciences and Systems (CISS)}, pp. 1--6, 2015.

\bibitem{Zhang:2015hn}
X.~Zhang, N.~Li, and A.~Papachristodoulou, ``{Achieving real-time economic
  dispatch in power networks via a saddle point design approach},'' in
  \emph{Power {\&} Energy Society General Meeting, 2015 IEEE}.\hskip 1em plus
  0.5em minus 0.4em\relax IEEE, 2015, pp. 1--5.

\bibitem{Kasis:2016vn}
A.~Kasis, E.~Devane, and I.~Lestas, ``{Primary frequency regulation with
  load-side participation: stability and optimality},'' \emph{arXiv.org}, 2016.

\bibitem{Boyd:2004cv}
S.~Boyd and L.~Vandenberghe, \emph{{Convex Optimization}}.\hskip 1em plus 0.5em
  minus 0.4em\relax Cambridge University Press, Mar. 2004.

\bibitem{Kirschen:2004ja}
D.~S. Kirschen and G.~Strbac, \emph{{Fundamentals of Power System Economics}},
  ser. Kirschen/Power System Economics.\hskip 1em plus 0.5em minus 0.4em\relax
  Chichester, UK: John Wiley {\&} Sons, Oct. 2004.

\bibitem{Miller:2011tm}
N.~W. Miller, M.~Shao, and S.~Venkataraman, ``{Frequency Response Study },''
  General Electric, Tech. Rep., Nov. 2011.

\bibitem{Doyle:1989kf}
J.~C. Doyle, K.~Glover, P.~P. Khargonekar, and B.~A. Francis, ``{State-space
  solutions to standard H$_{2}$ and H$_{\infty}$ control problems},''
  \emph{Automatic Control, IEEE Transactions on}, vol.~34, no.~8, pp. 831--847,
  Aug. 1989.

\bibitem{Tegling:2015ef}
E.~Tegling, B.~Bamieh, and D.~F. Gayme, ``{The Price of Synchrony: Evaluating
  the Resistive Losses in Synchronizing Power Networks},'' \emph{IEEE
  Transactions on Control of Network Systems}, vol.~2, no.~3, pp. 254--266,
  2015.

\bibitem{Poolla:2015vq}
B.~K. Poolla, S.~Bolognani, and F.~D{\"o}rfler, ``{Optimal Placement of Virtual
  Inertia in Power Grids},'' Oct. 2015.

\bibitem{Tegling:2016wna}
E.~Tegling, M.~Andreasson, J.~W. Simpson-Porco, and H.~Sandberg, ``{Improving
  performance of droop-controlled microgrids through distributed PI-control},''
  Jan. 2016.

\bibitem{Paganini:2009fk}
F.~Paganini and E.~Mallada, ``{A Unified Approach to Congestion Control and
  Node-Based Multipath Routing},'' \emph{Networking, IEEE/ACM Transactions on},
  vol.~17, no.~5, pp. 1413--1426, 2009.

\bibitem{khalil2002nonlinear}
H.~K. Khalil, \emph{{Nonlinear systems}}, 3rd~ed.\hskip 1em plus 0.5em minus
  0.4em\relax Prentice Hall, 2002.

\bibitem{chow1992toolbox}
J.~H. Chow and K.~W. Cheung, ``{A toolbox for power system dynamics and control
  engineering education and research},'' \emph{Power Systems, IEEE Transactions
  on}, vol.~7, no.~4, pp. 1559--1564, 1992.

\end{thebibliography}

\end{document}